\documentclass[a4paper,11pt]{amsart}
\usepackage{amsmath,amsfonts,amsthm,amssymb,enumerate}
\usepackage{graphicx}
\usepackage{comment}
\numberwithin{equation}{section}

\newtheorem{theorem}{Theorem}[section]

\newtheorem{lemma}[theorem]{Lemma}

{
\theoremstyle{definition}
\newtheorem{definition}[theorem]{Definition}

\newtheorem{remark}[theorem]{Remark}
}

\title{Deformations of dotted graphs consisting of standard circles}
\author{Inasa Nakamura}
\address{Department of Mathematics, Information Science and  Engineering, \newline
Saga University, \newline 
1 Honjomachi, Saga, 840-1153, Japan}
\email{inasa@cc.saga-u.ac.jp}
\subjclass[2020]{Primary: 05C10, Secondary: 57K10}
\keywords{graph; deformation; lattice polytope; partial matching}

\begin{document}  
\begin{abstract}
Dotted graphs are certain finite graphs with vertices of degree 2  called dots in the $xy$-plane $\mathbb{R}^2$, and a dotted graph is said to be admissible if it is associated with a lattice polytope in $\mathbb{R}^2$ each of whose edge is parallel to the $x$-axis or the $y$-axis.  
A dotted graph is said to be reducible if certain types of  deformations are applicable. In this paper, we investigate the reducibility of admissible dotted graphs in certain simple forms  consisting of standard circles.  
\end{abstract}
\maketitle

\section{Introduction}\label{sec1}

Let $\mathbb{R}^2$ be the $xy$-plane. A finite graph $\Gamma$ in $\mathbb{R}^2$ is a {\it dotted graph} \cite{N2} if each edge has an orientation and each vertex is of degree 2 or degree 4, satisfying that around each vertex of degree 2 (respectively 4), the edges (respectively, each pair of diagonal edges) have a coherent orientation. We denote each vertex of degree 2 by a small black disk, called a {\it dot}. 

The notion of a dotted graph was introduced to investigate lattice polytopes \cite{N}. We treat lattice polytopes as in Figure \ref{Fig2}(1); see \cite{N} for precise definition. 
 In order to treat admissible dotted graphs, it suffices to define  lattice polytopes in terms of graphs \cite{Diestel}, that is a special version of lattice polytopes in \cite{N}: in this paper, we give definition as follows. 
 A {\it lattice polytope} 
 is a dotted graph equipped with vertices of degree 2 called {\it $X$ marks} on edges avoiding dots and crossings such that as immersed oriented circles, dots and $X$ marks appear alternately and the oriented segment from each dot and to the next $X$ mark (respectively, from each $X$ mark to the next dot) is parallel to the $x$-axis (respectively, the $y$-axis). 
For a lattice polytope $P$, when we delete $X$ marks and consider $P$ as a planar graph up to isotopies, we obtain a dotted graph, which will be called the {\it dotted graph associated with a lattice polytope $P$}; see Figure \ref{Fig2}.
We say that a dotted graph $\Gamma$ is {\it admissible} if there exists a lattice polytope $P$ with which $\Gamma$ is associated. 

\begin{figure}[ht]
\centering
\includegraphics*[height=4cm]{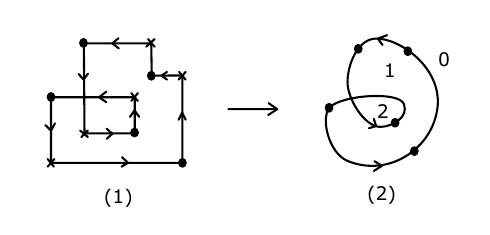}
\caption{(1) A lattice polytope and (2) the associated dotted graph. The $x$-direction is the vertical direction. } 
\label{Fig2}
\end{figure}

In \cite{N}, we investigated partial matchings, using lattice polytopes. 
A partial matching is a finite graph consisting of edges and vertices of degree zero or one, with a finite vertex set; partial matchings are used to investigate structures of polymers such as RNA secondary structures; see \cite{Reidys}. 
In \cite{N}, with a motivation to investigate partial matchings, we introduced the lattice presentation of a partial matching, and the lattice polytope associated with a pair of partial matchings. Further, in \cite{N}, we treated transformations of lattice presentations and lattice polytopes, and the area of a transformation, and the reduced graph of a lattice polytope; see \cite{N} for details. 
In \cite{N2}, we refined the notion of a reduced graph by introducing the notion of a dotted graph, and we introduced the notion of \lq\lq deformations'', in particular, \lq\lq good deformations'' of dotted graphs. These notions were introduced with the aim to establish  tools similar to knot diagrams \cite{Kawauchi} or chart description of surface braids \cite{Kamada}. 
A lattice polytope is associated with an admissible dotted graph. A dotted graph is said to be reducible if a deformation is applicable. 
 In \cite{N2}, we investigated deformations of dotted graphs, and 
 relation between deformations of admissible dotted graphs and transformations of lattice polytopes, using reduced graphs of dotted graphs. 
 
In this paper, with a motivation to investigate forms of reduced graphs,  
we investigate the reducibility of admissible dotted graphs in certain simple forms. We consider dotted graphs given as certain  combinations of components each of which is in the form of a unit circle (standard circle). Using combinatorial arguments, we show that admissible dotted graphs in such forms are reducible. 

We call a subgraph $C$ of a dotted graph $\Gamma$ a {\it circle component} if, regarding $\Gamma$ as immersed circles, $C$ is an  embedded circle. 
For a circle component $C$, we also denote by $C$ the disk bounded by $C$. 
In particular, 
for circle components $C$ and $C'$, we denote by $C \cap C'$ 
the intersection of the disks $C$ and $C'$, or more precisely, the disk $D \cap D'$, 
where $D$ (respectively $D'$) is the disk bounded by $C$ (respectively $C'$).  

\begin{theorem}\label{prop4-1}
Let $\Gamma$ be an admissible dotted graph such that $\Gamma$ consists of a pair of circle components $C_1$ and $C_2$ and the intersection of the bounded disks $C_1$ and $C_2$ forms a disk that is properly included in each bounded disk $C_i$ $(i=1,2)$. 
Then, $\Gamma$ is reducible by good deformations. The dotted graph $\Gamma$ has a good reduced graph that is the empty graph. 
\end{theorem}

\begin{theorem}\label{prop4-4}
Let $\Gamma$ be an admissible dotted graph such that $\Gamma$ consists of a finite number of circle components $C_1, \ldots, C_n$ and the bounding disks, also denoted by $C_1, \ldots, C_n$, satisfy  that $C_i \cap C_j$ is a disk  properly included in  each disk $C_i$ and  $C_j$ for $|i-j|=1$, and $C_i \cap C_j =\emptyset$ for $|i-j|>1$. We also include the case when $C_{1}=C_{n+1}$. 
Then, $\Gamma$ is reducible by good deformations. 
\end{theorem}

\begin{theorem}\label{prop4-5}
Let $\Gamma$ be an admissible dotted graph such that $\Gamma$ consists of three circle components $C_1, C_2, C_3$ such that in the sense of bounded disks, the intersection  $C_i \cap C_j$ $(i\neq j \in \{1,2,3\})$ of each pair of disks is a disk properly included in each disk $C_i$ and $C_j$, and the intersection $C_1 \cap C_2 \cap C_3$ of the disks is a disk. 
Then, $\Gamma$ is reducible by good deformations. 
\end{theorem}

The paper is organized as follows. 
In Section \ref{sec2}, we review dotted graphs and their deformations. 
In Section \ref{sec-lem}, as preliminaries, we show two lemmas. 
In Section \ref{sec3}, we show 
Theorems \ref{prop4-1}, \ref{prop4-4} and \ref{prop4-5}.

\section{Dotted graphs and their deformations}\label{sec2}
In this section, we review the notion of a dotted graph and  deformations of dotted graphs \cite{N2}. 
 
\begin{definition}[\cite{N2}]\label{def2-1}
Let $\Gamma$ be a finite graph in $\mathbb{R}^2$. Then $\Gamma$ is a {\it dotted graph} if each edge has an orientation and each vertex is of degree 2 or degree 4, satisfying the following conditions. 
\begin{enumerate}
\item
Around each vertex of degree 2, the edges have a coherent orientation. We call the vertex a {\it dot}, and we denote it by a small black disk.  

\item
Around each vertex of degree 4, each pair of diagonal edges has a coherent orientation. We call the vertex a {\it crossing}.

\end{enumerate}

We regard edges connected by vertices of degree 2 as an edge with several dots. We call an edge or a part of an edge an {\it arc}. 
We regard $\Gamma$ as immersed oriented circles, and we assign each region determined by $\Gamma$ its rotation number. 
Here, for an immersion $f: S^1_1 \sqcup S^1_2 \sqcup \ldots \sqcup S^1_m \to \mathbb{R}^2$ $(S^1_i=S^1)$ and a region $R \subset \mathbb{R}^2$ determined by the image of $f$, the {\it rotation number} of $f$ with respect to $R$ is the sum of rotation numbers of $g_i: S^1_i \to \mathbb{R}/2 \pi \mathbb{Z}$ $(i=1, \ldots,m)$ that maps $x \in S^1_i$ to the argument of the vector from a fixed interior point of $R$ to $f(x)$. Here, the {\it rotation number} of a map $g: S^1=\mathbb{R}/2\pi \mathbb{Z} \to \mathbb{R}/2 \pi \mathbb{Z}$ is defined by 
$(G(x)-x)/2\pi$, where $G: \mathbb{R} \to \mathbb{R}$ is the lift of $g$ and $x \in S^1$.
Two dotted graphs are {\it equivalent} if they are related by an ambient isotopy of $\mathbb{R}^2$. 

\end{definition}

 Let $\Gamma$ be a dotted graph. 
We call a subgraph of $\Gamma$ a {\it circle component} if it forms an embedded circle in the set of immersed circles determined by $\Gamma$. And we call a subgraph of $\Gamma$ a {\it loop component} if it is contained in one of immersed circles determined by $\Gamma$ and it forms a simple closed path whose base point is a crossing (self-intersection point). 
For a dotted graph $\Gamma$ that can be presented by $\Gamma=\Gamma_1 \cup \Gamma_2$, where $\Gamma_1$ and $\Gamma_2$ are dotted graphs, we say that a region $R_1$ of $\Gamma_1$ is an {\it overlapping region} or a region $R_2$ of $\Gamma_2$ is an {\it overlapped region}, if $R_1 \cap R_2 \neq \emptyset$ and the label of the region of $\Gamma_1$ containing $R_2\backslash R_1$ is zero; see \cite{N2}.  
See Figure \ref{20240531-3} for an example of an overlapping region.

Let $\Gamma$ and $\Gamma'$ be dotted graphs. 
Then, 
$\Gamma'$ is obtained from $\Gamma$ by 
a {\it dotted graph deformation} or simply a  {\it deformation} I, II, III or IV if $\Gamma$ is changed to $\Gamma'$ by a local move in a disk as shown in Figure \ref{Fig3}; see \cite[Definition 3.3]{N2} for  precise definition.

\begin{figure}[ht]
\centering
\includegraphics*[height=4.5cm]{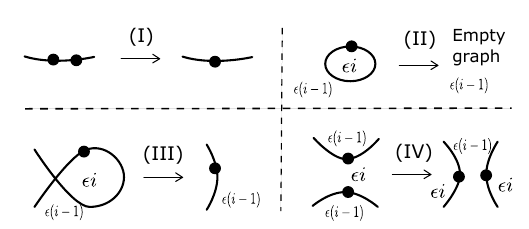}
\caption{Local deformations I--IV, where $\epsilon \in \{+1, -1\}$ and $i$ is a positive integer. A deformation IV is applicable when the arcs admit induced orientations. 
Deformations II, III, IV are applicable including the case when the regions are overlapped by several regions of the label $\epsilon$. }
\label{Fig3}
\end{figure}

Let $\Gamma$ be a dotted graph. We say that $\Gamma$ is  {\it reducible} if we can apply one of deformations I--IV. 
And we call a dotted graph $\Gamma'$ a {\it reduced graph} of 
$\Gamma$ if $\Gamma'$ is not reducible and $\Gamma'$ is obtained from $\Gamma$ by a finite sequence of deformations I--IV. We remark that there exist non-empty reduced graphs; see \cite[Figure 6.2]{N} for an example.

Further, we consider a specific deformation IV, called a {\it deformation IVa}, that satisfies one of the conditions (a1) and (a2) given in \cite[Definition 3.4]{N2}. In this paper, we use only (a1), as follows. We denote by $\mathcal{R}$ a deformation IV: 

\begin{enumerate}

 \item[(a1)]
The arcs involved in $\mathcal{R}$ are adjacent arcs of a crossing, where we ignore overlapping regions, such that $\mathcal{R}$ creates a loop component applicable of a deformation III. 

\end{enumerate}

We call deformations I, II, III and IVa {\it good deformations}, 
and we call a reduced graph of good deformations a {\it good reduced graph}, when we apply a deformation III after each deformation IVa to delete the loop component. 
See Figure \ref{20240531-3} for an example of a deformation IVa. 
See \cite{N2} for results concerning good deformations and transformations of lattice polytopes. 

\begin{figure}[ht]
\centering
\includegraphics*[height=3.5cm]{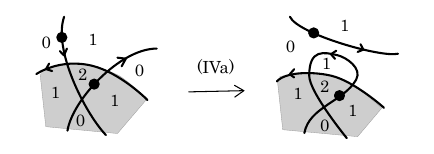}
\caption{Example of a deformation IVa (a1), where the shadowed region in each figure is an overlapping region. }
\label{20240531-3}
\end{figure}

\section{Lemmas}\label{sec-lem}
As preliminaries, we show lemmas. 
We recall that a circle component of a dotted graph is the set of arcs with dots and connecting crossings that bounds a disk such that any pair of arcs connected by a crossing is a pair of diagonal arcs. And we call a set of arcs with dots and connecting crossings a {\it bigon component} or simply a {\it bigon} if it bounds a disk and two pairs of arcs connected by a crossing are pairs of adjacent arcs and 
any other pair of arcs connected by a crossing is a pair of diagonal arcs, and moreover around each crossing connecting the adjacent arcs forming the component, the bounded disk does not contain the other arcs. 

\begin{lemma}\label{lem4-3}
Let $\Gamma$ be an admissible dotted graph. 
Let $C$ be a bigon component such that its forming arcs do not admit coherent orientations as a circle. Then, $C$ has at least one dot. 
\end{lemma}

\begin{proof}
By considering the corresponding lattice polytope, we see that there are at least one dot and X mark on the bigon. Thus $C$ has at least one dot; see Figure \ref{Fig6}. 
\end{proof}

\begin{figure}[ht]
\centering
\includegraphics*[height=2.5cm]{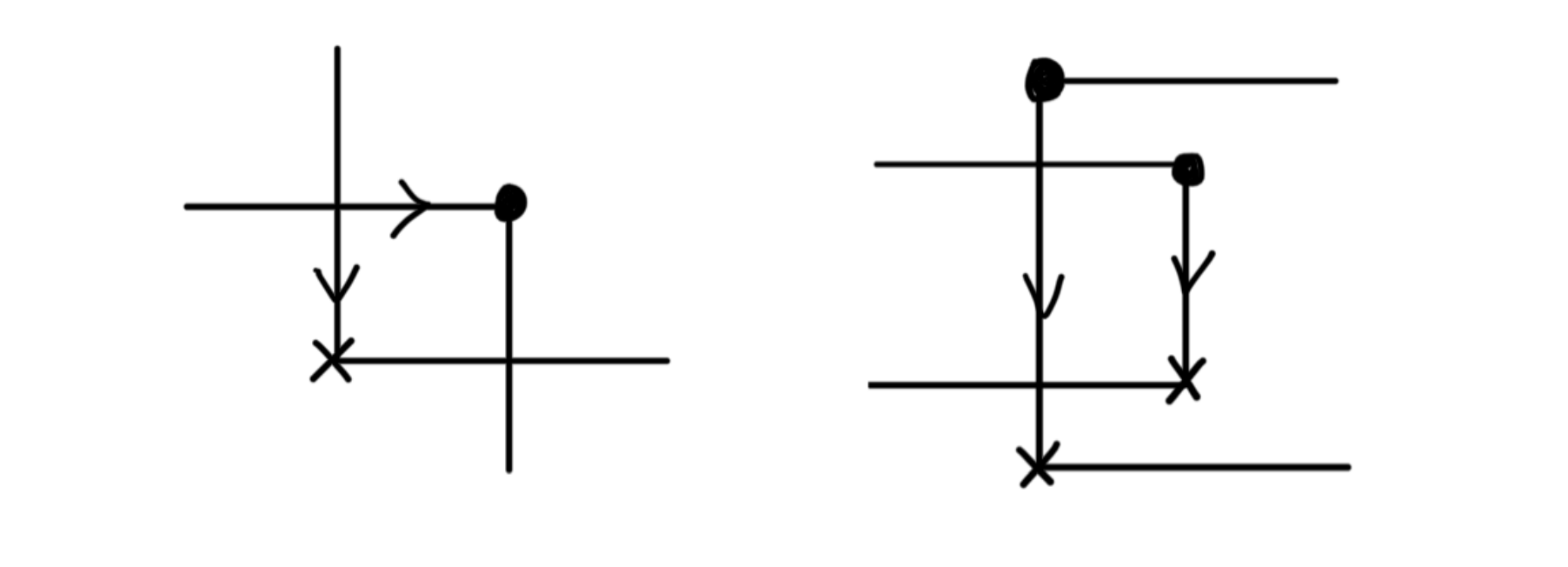}
\caption{The part of a lattice polytope is essentially as this figure, that realizes a bigon component whose arcs do not admit coherent orientations as a circle. }
\label{Fig6}
\end{figure}

Let $\Gamma$ be a dotted graph. 
We call a set $G$ of arcs with dots and connecting crossings a {\it  crossing-including component} if its union is homeomorphic to a circle $C$ and, when we ignore crossings of $G$ connecting diagonal arcs in $G$, around each of the other crossings of $G$, the closed disk bounded by $C$ contains the four arcs forming the crossing. In particular, we call a crossing-including component an {\it outermost component} if it bounds a disk $D$ such that there is a regular neighborhood $N(D)$ of $D$ such that $N(D) \cap \Gamma=D \cap \Gamma$. 
We remark that a loop/bigon component is not a crossing-including component. 

\begin{lemma}\label{lem4-2}
Let $\Gamma$ be an admissible dotted graph. Then, a crossing-including component has at least two dots. 
In particular, a circle component has at least two dots.
\end{lemma}

\begin{proof}
Let $P$ be a lattice polytope associated with $\Gamma$. 
We consider a crossing-including component of $\Gamma$, and 
we denote by $G$ the corresponding subgraph of $P$. Further, we ignore crossings of $G$ connecting diagonal arcs in $G$, regarding the two arcs as one arc. 
The subgraph $G$ consists of edges and crossings (and dots and $X$ marks) such that in a neighborhood of each crossing, the four arcs of $P$ forming the crossing are included in the closed disk bounded by $G$; this implies that the inner angle of $G$ at each crossing is $3\pi/2$. Further, we see that since $G$ bounds a disk, we have $n$ inner angles of $3\pi/2$ and $n+4$ inner angles of $\pi/2$ for some non-negative integer $n$. Hence we have at least 4 inner angles of $\pi/2$ each of which is formed by a pair of edges connected by a vertex of degree 2, and among the vertices at least two are dots and at least two are $X$ marks. Thus, the dotted graph $\Gamma$ has at least two dots. 
\end{proof}

\section{Admissible dotted graphs that are reducible}\label{sec3}

In this section, we show Theorems \ref{prop4-1}, \ref{prop4-4} and  \ref{prop4-5}. 
For a circle component $C$, we also denote by $C$ the union of regions forming the disk bounded by $C$. We regard the disk as a region (overlapped by some regions) and we call the label of the region $C$ the {\it label of the circle component}.  
In particular, 
for circle components $C$ and $C'$, we denote by $C \cap C'$ 
the bigon component or the disk bounded by the bigon component. 
  
For a set $X$ of arcs $a_1, \ldots, a_k$ with dots and connecting crossings, we denote $X$ by $\cup_{i=1}^k a_i$, and when $X$ bounds a disk, we denote the disk by $|a_1 \cdots a_k|$. 
We mostly take as $X$ a circle/bigon component.

\begin{proof}[Proof of Theorem \ref{prop4-1}]
We denote the circle component by $C_1$ and $C_2$, and we denote the arcs by $a_1, \ldots, a_4$ such that $C_1$ consists of $a_1$ and $a_3$, and $C_2$ consists of $a_2$ and $a_4$, and the bigon $C_1 \cap C_2$ consists of $a_2$ and $a_3$, respectively; see Figure \ref{Fig7}. 
We remark that the label of the region $C_1 \cap C_2$ is the total sum of the labels of regions $C_1$ and $C_2$. 

(Case 1) When $C_1$ and $C_2$ both have the label $\epsilon$ $(\epsilon \in \{+1, -1\})$. 
In this case, by deformations II, we can deform $\Gamma$ to the empty graph.  
 
(Case 2) When $C_1$ and $C_2$ have the labels $\epsilon$ and $-\epsilon$, respectively $(\epsilon \in \{+1, -1\})$. 
It suffices to show for the case when $\epsilon=+1$. Then the labels of regions $C_1\backslash (C_1 \cap C_2)$, $C_1 \cap C_2$, $C_2\backslash (C_1 \cap C_2)$ are $+1, 0, -1$, respectively. 
Assume that $\Gamma$ is not reducible. 
By Lemma \ref{lem4-3}, the bigon $C_1 \cap C_2=a_2\cup a_3$ has a dot. 
By symmetry of the graph, we can assume 
that $a_2$ has a dot. Then, since $C_1$ has a dot by Lemma \ref{lem4-2}, either $a_1$ or $a_3$ has a dot. If $a_1$ has a dot, then we can apply a deformation IVa for the region $|a_1a_2|$, which is a contradiction. Hence, $a_3$ has a dot. Then, since the boundary of the closure of the disk $C_1 \cup C_2$ (the outermost component) has a dot by Lemma \ref{lem4-2}, $a_1$ or $a_4$ has a dot. In both cases, we can apply a deformation IVa for the region $|a_1a_2|$ or $|a_3a_4|$, which is a contradiction. Thus, $\Gamma$ is reducible. 
In these cases, by applying a deformation IVa and then deformations III and II, the graph is deformed to the empty graph; thus we see that the empty graph is a good reduced graph of $\Gamma$.
\end{proof}

\begin{figure}[ht]
\centering
\includegraphics*[height=3.5cm]{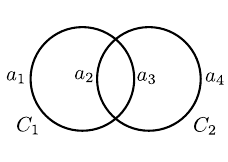}
\caption{Arcs $a_1, a_2, a_3, a_4$ in the proof of Theorem  \ref{prop4-1}. }
\label{Fig7}
\end{figure}

We show Theorem \ref{prop4-4}. 
When $C_1 \cap C_n=\emptyset$, we say that $\Gamma$ is {\it in a straight form}, and when $C_1 \cap C_n \neq \emptyset$ and $C_1=C_{n+1}$, we say that $\Gamma$ is {\it in a ring form}. 

\begin{proof}[Proof of Theorem \ref{prop4-4}]
Assume that $\Gamma$ is not reducible. 
By Theorem \ref{prop4-1}, we see that $n>2$. 

(Step 1) 
First we discuss the case when there exist $C_i$, $C_{i+1}$, and $C_{i+2}$ that have the same label $\epsilon$ $(\epsilon \in \{+1, -1\})$. 
In this case, we can apply a deformation II to $C_{i+1}$. Hence this case does not occur. 

(Step 2) 
Next we discuss the case when there exist $C_i$ and $C_{i+1}$ that have the same label $\epsilon$ $(\epsilon \in \{+1, -1\})$. 
If $\Gamma$ is in a straight form and $C_i=C_1$ or $C_{i+1}=C_n$, we can apply a deformation II and remove $C_i$ or $C_{i+1}$. 
Hence, we see that $C_i$ and $C_{i+1}$ are the middle of a sequence of circle components $C_{i-1}$, $C_i$, $C_{i+1}$, and $C_{i+2}$, where we include the case $C_{i-1}=C_{i+2}$ for the ring form. Further, by the argument in Step 1, we see that $C_{i-1}$ and $C_{i+1}$ have the label $-\epsilon$. 

We denote by $a_1, a_2, a_3, a_4$ the arcs such that in the sense of regions, 
\begin{eqnarray*}
  && C_{i-1} \cap C_{i}=|a_1 a_2|, a_1 \subset  C_i, a_2 \subset C_{i-1}, \\
&& C_{i+1} \cap C_{i+2}=|a_3 a_4|, a_3 \subset  C_{i+2}, a_4 \subset C_{i+1}. 
\end{eqnarray*}
By Lemma \ref{lem4-3}, $a_1 \cup a_2$ has a dot. Consider the case when $a_2$ has a dot. Then, dots of the circle component $C_i$ must be in $a_1$, because otherwise we can apply a deformation IVa. 
Since $C_i$ has a dot by Lemma \ref{lem4-2}, the case when $a_2$ has a dot is included in the case when $a_1$ has a dot, so we see that $a_1$ must have a dot. 
Similarly, we see that there is a dot in $a_4$. 

(Step 3)
We consider the sequence of bigons $C_{i} \cap C_{i+1}$ ($i=1,2, \ldots$). We call the arc of $C_i$ (respectively $C_{i+1}$) in $C_{i} \cap C_{i+1}$ the {\it right arc} (respectively {\it left arc}). We denote by a sequence of letters $(s_1, s_2, \ldots)$ the sequence of information of existence of dots in the bigons, called the {\it sequence of dot existence}, obtained as follows. If we are certain that the $j$th bigon has a dot in its left arc (respectively right arc), then we write $s_j=l$ (respectively $r$), and if we are certain that it has dots in both left and right arcs, then we write $s_j=b$, and if we are not certain that the bigon has dots, then we right $e$. We remark that when we write $s_j=l$ or $s_j=r$, the other arc may have dots. In the above situation of $C_{i-1}$, $C_i$, $C_{i+1}$, and $C_{i+2}$ such that $C_i$, $C_{i+1}$ have the label $\epsilon$ and the others have the label $-\epsilon$, the sequence of dot existence is $(l, e, r)$. 

The regions bounded by bigons $C_{i} \cap C_{i+1}$ $(i=1,2,\ldots)$ have labels in $\{0, \pm 2\}$. We consider the sequence of the labels, called the {\it sequence of labels of the bigons}. Since the label of $C_{i} \cap C_{i+1}$ is zero (respectively $\epsilon 2$, where $\epsilon \in \{+1, -1\}$) if the labels of $C_i$ and $C_{i+1}$ are $(1, -1)$ or $(-1, 1)$ (respectively $(\epsilon, \epsilon)$), the sequence of labels of the bigons is a sequence consisting of $0, \pm 2$ that does not contain $(2,-2)$ or $(-2,2)$.  
Further, by the arguments  in Steps 1 and 2, we see that the sequence does not contain $(-2, -2)$ or $(2,2)$, and the sequence begins and ends with $0$ if $\Gamma$ is in the straight form. 
And the sequence of labels of the bigons $(0,\pm 2,0)$ is associated with the sequence of dot existence $(l, e, r)$, and the sequence of labels of the bigons $(0, \epsilon 2,0, \delta 2, 0)$ $(\epsilon, \delta \in \{+1, -1\})$ is associated with the sequence of dot existence $(l, e, b,e,r)$. 

If two adjacent bigons $X_1$ and $X_2$ with the label zero have a dot in the right arc of $X_1$ and a dot in the left arc of $X_2$, then we can apply a deformation IVa. Hence, the sequence of dot existence must not have a subsequence $(r,l)$ or $(b,l)$ or $(r,b)$ associated with a subsequence $(0,0)$ of labels of the bigons. Since a bigon with the label zero must have a dot by Lemma \ref{lem4-3}, we see that 
the possible sequences of dot existence are as follow.  
In the straight form, 
\begin{equation}\label{eq0813-1}
 (l,l,\ldots, l, l,e,b,e,b,e,\ldots, b,e,r,r,\ldots, r) 
\end{equation}
associated with the sequence of labels of bigons 
\begin{equation*}
 (0,0,\ldots, 0, \epsilon_1 2,0,\epsilon_2 2,0,\ldots, \epsilon_m 2, 0, 0,\ldots, 0), 
\end{equation*}
where $\epsilon_1, \ldots, \epsilon_m \in \{+1, -1\}$ for some positive integer $m$, 
or 
\begin{equation}\label{eq0916}
 (l,l,\ldots, l) \text{ or }  (r,r,\ldots, r) 
\end{equation}
associated with the sequence of labels of bigons
\begin{equation*}\label{eq0917}
 (0,0,\ldots, 0).  
\end{equation*}
And in the ring form, 
\begin{equation}\label{eq0813-2}
 (b,e,b,e, \ldots, b,e) \text{ or } (e,b,e,b, \ldots, e,b)
\end{equation}
 associated with the sequence of labels of bigons  
\begin{equation*}
 (0, \epsilon_1 2,0,\epsilon_2 2,\ldots, 0,\epsilon_m 2)
\text{ or } (\epsilon_1 2,0,\epsilon_2 2,0,\ldots, \epsilon_m 2,0), 
\end{equation*}
where $\epsilon_1, \ldots, \epsilon_m \in \{+1, -1\}$ for some positive integer $m$, or (\ref{eq0916}) associated with (\ref{eq0917}).

When we have (\ref{eq0813-1}), we have at least one $l$ and at least one $r$. 
We consider the case when (\ref{eq0813-1}) contains $(l,l)$. 
By Lemma \ref{lem4-2} we have a dot in the first circle component $C_1$. The dot is either in the right arc of the first  bigon or in the other arc of $C_1$. In both cases, we can apply a deformation IVa. Thus, in this case the dotted graph is reducible. 
By a similar argument, we see that for the case when (\ref{eq0813-1}) contains $(r,r)$, the dotted graph is reducible. 
The rest is the case $(l, e, r)$. In this case, since the circle component $C_1$ has a dot, the right arc of the first bigon has a dot; because if the other arc has a dot we can apply a deformation IVa. 
Similarly, the left arc of the third bigon has a dot, and the sequence of dot existence becomes $(b, e, b)$. By Lemma \ref{lem4-2}, we have a dot $p$ in an arc of the outermost component. When $p$ is in $C_1$ (respectively $C_2$), then we can apply a deformation IVa between $p$ and the dot in the left arc (respectively the right arc) of the first bigon. The other cases are also reducible by a similar argument. Thus the case $(l, e, r)$ is also reducible; hence the case of (\ref{eq0813-1}) is reducible. 

When we have (\ref{eq0813-2}), then we have dots in every arc of bigons with the label zero, and the arcs come from all circle components.  By Lemma \ref{lem4-2}, we have a dot $p$ in an arc of the outermost component. Let $C_j$ be the circle component containing $p$. Then we can apply a deformation IVa between $p$ and the dot in the right arc of the $(j-1)$th bigon or the dot in the left arc of the $j$th bigon. Thus the dotted graph is reducible. 

By a similar argument, we see that the case of (\ref{eq0916}) for the straight/ring form is also reducible.  
\end{proof}

\begin{proof}[Proof of Theorem \ref{prop4-5}]
Assume that $\Gamma$ is not reducible. 
We denote by $a_1, \ldots, a_7$ the arcs such that in the sense of regions, 
\begin{eqnarray*}
&& a_1, a_2, a_7 \subset C_2, \ a_3, a_4 \subset C_1, \ a_5, a_6 \subset C_3, \\
&& C_1 \cap C_2=|a_1 a_2 a_3 a_4|, \\
&& C_2 \cap C_3=|a_5 a_6 a_7 a_1|\\
&& C_1 \cap C_2 \cap C_3 =|a_1 a_5 a_4|.
\end{eqnarray*}
Further, we put $a_8=C_2 \backslash  (a_1 \cup a_2 \cup a_7)$; see Figure \ref{Fig8}. 

\begin{figure}[ht]
\centering
\includegraphics*[height=4cm]{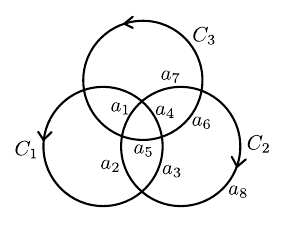}
\caption{The arcs $a_1, \ldots, a_8$ in the proof of Theorem  \ref{prop4-5}; the orientations are when $\epsilon=+1$. }
\label{Fig8}
\end{figure}

If $C_1, C_2, C_3$ have the same label, then we can apply  deformations II to deform $\Gamma$ to an empty graph. 

Hence, two circle components have the label $\epsilon$ and the other circle component have the label $-\epsilon$. $(\epsilon \in \{+1, -1\})$. We assume that $C_1$ and $C_3$ have the label $\epsilon$ and $C_2$ has the label $-\epsilon$. 

By Lemma \ref{lem4-3}, the bigon $C_1 \cap C_2=|a_1 a_2 a_3 a_4|$ has a dot. 
In the following argument, we show that neither $a_2$ nor $a_3$ has dots. 

(Case 1) Assume that $a_2$ has a dot. 
There is a dot in some arc of the circle component $C_1$. 
Since we must have a situation that a deformation IVa is not applicable, we see that the dot is in $a_3$ or $a_4$. 
If $a_3$ has a dot, then, since there is a dot $p$ in some arc of the outermost component of $C_1 \cup C_2$ by Lemma \ref{lem4-2}, we see that $p$ is in $a_7$; otherwise a deformation IVa is applicable. Then, we have dots in each of $a_2, a_3, a_7$. Since there is a dot in some arc of the outermost component of $C_1 \cup C_2 \cup C_3$, we can apply a deformation IV, which is a contradiction; hence $a_4$ has a dot. So the case when $a_3$ has a dot is included in the case when $a_4$ has a dot. 

We have dots in each of $a_2$ and $a_4$. Since a bigon $C_2 \cap C_3$ has a dot, the dot is either in $a_1$ or $a_6$ or $a_7$; if $a_5$ has a dot, we can apply a deformation IVa between $a_4$ and $a_5$. 

(Case a) 
If $a_1$ has a dot, then, since the circle component $C_3$ has a dot, $a_6$ has a dot; if there is a dot $p$ in $a_5$, then we can apply a deformation IVa between $a_5$ and $a_4$, and if there is a dot $p$ in either of the arcs of $C_3$ forming the outermost component of  $C_2 \cup C_3$, then we can apply a deformation IVa between the arc of $p$ and $a_1$. 
So this is included in the case when $a_6$ has a dot.

(Case b)  
If $a_6$ has a dot, then, since we have a dot in the outermost component $C_1 \cup C_2$, $a_7$ also has a dot. Then, we can apply a deformation IVa, because the outermost component of $C_1 \cup C_2 \cup C_3$ has a dot. 

(Case c)
If $a_7$ has a dot, then,  since we have a dot in the outermost component of $C_1 \cup C_3$, $a_3$ or $a_6$ has a dot. In both cases, we can apply a deformation IVa, because the outermost component of $C_1 \cup C_2 \cup C_3$ has a dot. 
Thus, (Case 1) does not occur. 

(Case 2) Assume that $a_3$ has a dot. 
Since the circle component $C_2$ has a dot and $a_2$ does not have a dot by the argument in (Case 1), we see that either $a_1$ or $a_7$ has a dot. If $a_1$ has a dot, since the outermost component of $C_1 \cup C_2$ has a dot, we see that $a_7$ has a dot; the other cases are applicable of a deformation IVa. Hence this case is included in the case when $a_7$ has a dot, and we see that $a_7$ must have a dot. 
We have a dot in each of $a_3$ and $a_7$. 
Since the outermost component of $C_2 \cup C_3$ has a dot, we see that $a_2$ has a dot; the other cases are applicable of a deformation IV. Then, since the outermost component of $C_1 \cup C_2 \cup C_3$ has a dot, we can apply a deformation IVa. 
Hence, (Case 2) does not occur. 

Thus, by the arguments in (Case 1) and (Case 2), neither $a_2$ nor $a_3$ has dots. 
By symmetry of $\Gamma$, we see that neither $a_6$ nor $a_7$ 
has dots. 
This implies that, since we have dots in each of the bigons $C_1 \cap C_2$ and $C_2 \cap C_3$, $a_1$ has a dot: if $a_1$ does not have a dot, then each of $a_4$ and $a_5$ has a dot, and we can apply a deformation IV between the dots in $a_4$ and $a_5$. 
Then, since the outermost component of $C_1 \cup C_2 \cup C_3$ has a dot, $a_8=C_2 \backslash  (a_1 \cup a_2 \cup a_7)$ has a dot. 
Since the circle component $C_1$ has a dot, $a_4$ has a dot. 
Similarly, since the circle component $C_3$ has a dot, $a_5$ has a dot.
Then, we can apply a deformation IVa between the dots in $a_4$ and $a_5$, which is a contradiction. 
Thus, $\Gamma$ is reducible. 
\end{proof}

\section*{Acknowledgements}
The author was partially supported by JST FOREST Program, Grant Number JPMJFR202U.


\begin{thebibliography}{0}

  \bibitem{Diestel}
R. Diestel, {\em Graph Theory}; Graduate Texts in Mathematics 173, American Mathematical Society, Springer, 2010. 

\bibitem {Kamada}
S. Kamada, {\it Braid and Knot Theory in Dimension Four}, Math. Surveys and Monographs 95, Amer. Math. Soc., 2002.


\bibitem{Kawauchi}
A. Kawauchi, {\it A Survey of Knot Theory}, Birkh\"{a}user Verlag, Basel, 1996. 

\bibitem {N}
I. Nakamura, {\em Transformations of partial matchings}; Kyungpook Math. J. 61 (2021), No.2, 409-439. 

\bibitem{N2}
I. Nakamura, {\em Transformations of lattice polytopes and their associated dotted graphs}; arXiv: 2310.00218.
 
\bibitem{Reidys} 
C. M. Reidys, {\em Combinatorial Computational Biology of RNA. Pseudoknots and neutral networks}; Springer, New York, 2011.


\end{thebibliography}
\end{document}